\documentclass[12pt]{amsart}
\usepackage{amsmath,amssymb,amsbsy,amsfonts,latexsym,amsopn,amstext,cite,
                                               amsxtra,euscript,amscd,bm,mathabx}
\usepackage{url}
\usepackage[colorlinks,linkcolor=blue,anchorcolor=blue,citecolor=blue,backref=page]{hyperref}
\usepackage{color}
\usepackage{graphics,epsfig}
\usepackage{graphicx}
\usepackage{float} 
\usepackage[english]{babel}
\usepackage{mathtools}
\usepackage{todonotes}
\usepackage{url}
\usepackage[colorlinks,linkcolor=blue,anchorcolor=blue,citecolor=blue,backref=page]{hyperref}
\DeclareMathAlphabet{\mathmybb}{U}{bbold}{m}{n}

\usepackage[norefs,nocites]{refcheck}
\hypersetup{breaklinks=true}

\usepackage[norefs,nocites]{refcheck}
\usepackage[english]{babel}
\begin{document}

\newtheorem{thm}{Theorem}
\newtheorem{lem}[thm]{Lemma}
\newtheorem{claim}[thm]{Claim}
\newtheorem{cor}[thm]{Corollary}
\newtheorem{prop}[thm]{Proposition} 
\newtheorem{definition}[thm]{Definition}
\newtheorem{rem}[thm]{Remark} 
\newtheorem{question}[thm]{Open Question}
\newtheorem{conj}[thm]{Conjecture}
\newtheorem{prob}{Problem}
\newtheorem{Process}[thm]{Process}
\newtheorem{Computation}[thm]{Computation}
\newtheorem{Fact}[thm]{Fact}
\newtheorem{Observation}[thm]{Observation}

\newtheorem{lemma}[thm]{Lemma}

\newcommand{\GL}{\operatorname{GL}}
\newcommand{\SL}{\operatorname{SL}}
\newcommand{\lcm}{\operatorname{lcm}}
\newcommand{\ord}{\operatorname{ord}}
\newcommand{\Op}{\operatorname{Op}}
\newcommand{\Tr}{\operatorname{Tr}}
\newcommand{\Nm}{\operatorname{Nm}}

\numberwithin{equation}{section}
\numberwithin{thm}{section}
\numberwithin{table}{section}

\numberwithin{figure}{section}

\def\sssum{\mathop{\sum\!\sum\!\sum}}
\def\ssum{\mathop{\sum\ldots \sum}}
\def\iint{\mathop{\int\ldots \int}}

\def\wt {\mathrm{wt}}
\def\Tr {\mathrm{Tr}}

\def\SrA{\cS_r\(\cA\)}

\def\vol {{\mathrm{vol\,}}}
\def\squareforqed{\hbox{\rlap{$\sqcap$}$\sqcup$}}
\def\qed{\ifmmode\squareforqed\else{\unskip\nobreak\hfil
\penalty50\hskip1em\null\nobreak\hfil\squareforqed
\parfillskip=0pt\finalhyphendemerits=0\endgraf}\fi}

\def \ss{\mathsf{s}} 

\def \balpha{\bm{\alpha}}
\def \bbeta{\bm{\beta}}
\def \bgamma{\bm{\gamma}}
\def \blambda{\bm{\lambda}}
\def \bchi{\bm{\chi}}
\def \bphi{\bm{\varphi}}
\def \bpsi{\bm{\psi}}
\def \bomega{\bm{\omega}}
\def \btheta{\bm{\vartheta}}

\newcommand{\bfxi}{{\boldsymbol{\xi}}}
\newcommand{\bfrho}{{\boldsymbol{\rho}}}

 \def \xbar{\overline x}
  \def \ybar{\overline y}

\def\cA{{\mathcal A}}
\def\cB{{\mathcal B}}
\def\cC{{\mathcal C}}
\def\cD{{\mathcal D}}
\def\cE{{\mathcal E}}
\def\cF{{\mathcal F}}
\def\cG{{\mathcal G}}
\def\cH{{\mathcal H}}
\def\cI{{\mathcal I}}
\def\cJ{{\mathcal J}}
\def\cK{{\mathcal K}}
\def\cL{{\mathcal L}}
\def\cM{{\mathcal M}}
\def\cN{{\mathcal N}}
\def\cO{{\mathcal O}}
\def\cP{{\mathcal P}}
\def\cQ{{\mathcal Q}}
\def\cR{{\mathcal R}}
\def\cS{{\mathcal S}}
\def\cT{{\mathcal T}}
\def\cU{{\mathcal U}}
\def\cV{{\mathcal V}}
\def\cW{{\mathcal W}}
\def\cX{{\mathcal X}}
\def\cY{{\mathcal Y}}
\def\cZ{{\mathcal Z}}
\def\Ker{{\mathrm{Ker}}}

\def\NmQR{N(m;Q,R)}
\def\VmQR{\cV(m;Q,R)}

\def\Xm{\cX_{p,m}}

\def \A {{\mathbb A}}
\def \B {{\mathbb A}}
\def \C {{\mathbb C}}
\def \F {{\mathbb F}}
\def \G {{\mathbb G}}
\def \L {{\mathbb L}}
\def \K {{\mathbb K}}
\def \N {{\mathbb N}}
\def \PP {{\mathbb P}}
\def \Q {{\mathbb Q}}
\def \R {{\mathbb R}}
\def \Z {{\mathbb Z}}
\def \fS{\mathfrak S}
\def \fB{\mathfrak B}

\def\Fq{\F_q}
\def\Fqr{\F_{q^r}} 
\def\ovFq{\overline{\F_q}}
\def\ovFp{\overline{\F_p}}
\def\GL{\operatorname{GL}}
\def\SL{\operatorname{SL}}
\def\PGL{\operatorname{PGL}}
\def\PSL{\operatorname{PSL}}
\def\li{\operatorname{li}}
\def\sym{\operatorname{sym}}

\def\Mob{M{\"o}bius }

\def\fF{\EuScript{F}}
\def\M{\mathsf {M}}
\def\T{\mathsf {T}}

\def\e{{\mathbf{\,e}}}
\def\ep{{\mathbf{\,e}}_p}
\def\eq{{\mathbf{\,e}}_q}

\def\\{\cr}
\def\({\left(}
\def\){\right)}

\def\<{\left(\!\!\left(}
\def\>{\right)\!\!\right)}
\def\fl#1{\left\lfloor#1\right\rfloor}
\def\rf#1{\left\lceil#1\right\rceil}

\def\Tr{{\mathrm{Tr}}}
\def\Nm{{\mathrm{Nm}}}
\def\Im{{\mathrm{Im}}}

\def \oF {\overline \F}

\newcommand{\pfrac}[2]{{\left(\frac{#1}{#2}\right)}}

\def \Prob{{\mathrm {}}}
\def\e{\mathbf{e}}
\def\ep{{\mathbf{\,e}}_p}
\def\epp{{\mathbf{\,e}}_{p^2}}
\def\em{{\mathbf{\,e}}_m}

\def\Res{\mathrm{Res}}
\def\Orb{\mathrm{Orb}}

\def\vec#1{\mathbf{#1}}
\def \va{\vec{a}}
\def \vb{\vec{b}}
\def \vh{\vec{h}}
\def \vk{\vec{k}}
\def \vs{\vec{s}}
\def \vu{\vec{u}}
\def \vv{\vec{v}}
\def \vz{\vec{z}}
\def\flp#1{{\left\langle#1\right\rangle}_p}
\def\T {\mathsf {T}}

\def\sfG {\mathsf {G}}
\def\sfK {\mathsf {K}}

\def\mand{\qquad\mbox{and}\qquad}

\title[Rational approximation with radix constraints]
{Rational Approximation with digit-restricted denominators}

\author[Siddharth Iyer] {Siddharth Iyer}
\address{School of Mathematics and Statistics, University of New South Wales, Sydney, NSW 2052, Australia}
\email{siddharth.iyer@unsw.edu.au}

\begin{abstract} We show the existence of ``good'' approximations to a real number $\gamma$ using rationals with denominators formed by digits $0$ and $1$ in base $b$. We derive an elementary estimate and enhance this result by managing exponential sums.
\end{abstract}


\maketitle

\tableofcontents

\section{Introduction}

Given a set $H \subseteq \mathbb{N}$ of positive natural numbers, we are interested in whether we can obtain that for any $\gamma \in \R$ and $N\in \N$, one has
\begin{align*}
\min_{1 \leq n \leq N,\ n \in H} \|n \gamma\| \leq f(N),
\end{align*}
where $\|.\|$ is the distance to the nearest integer function and $f$ is some function decaying to zero. We say that the set $H$ is a $\mathcal{H}$-set if the above property is satisfied. Literature has discussed a weaker notion of \textit{Heilbronn} sets \cite[p. 35]{Montgomery}. A set $H \subseteq \N$ is said to be Heilbronn if for all $\varepsilon > 0$ and $\gamma \in \R$ there exists positive $h_{\gamma}(\varepsilon) \in H$ so that $\|h_{\gamma}(\varepsilon) \gamma\| < \varepsilon$.

A simple example of a Heilbronn set is $\mathbb{N}$, since for any $\gamma \in \R$ we can exploit the Dirichlet approximation Theorem to write
\begin{equation*}
\min_{1 \leq n \leq N} \|\gamma n\| \leq \frac{1}{N},
\end{equation*}
saying that it is a $\mathcal{H}-$set would be a stronger assertion.
If a set $H$ is discovered to be a $\mathcal{H}-$set along with a suitable decay function $f$ then for any $N \in \N$ and $\gamma \in \R,$ there exist $a$ and $b$ in the set $\{1,\ldots,N\}\cap H$ so that
\begin{align*}
\left|\gamma - \frac{a}{b}\right| \leq \frac{1}{bf(N)}.
\end{align*}
 With this interpretation of a $\mathcal{H}-$set, we discover how well rationals with denominators in $H$ approximate real numbers, ``how well'' being dependent on the quality of the decay function $f$.

One need not be content after possibly discovering that a certain set $H$ is a $\mathcal{H}-$set; discovering a stronger decay function or showing that the function is optimal (with notions of optimality) can be one avenue for the next research stage. In the case of square integers, Heilbronn  demonstrated that for any $\gamma \in \R$, one has 
\begin{equation}
\label{Squareref}
\min_{1 \leq n \leq N}\|\gamma n^2\| \leq C_{\theta} N^{-\theta},
\end{equation}
for all $\theta < \frac{1}{2}$, where $C_{\theta}$ is some real number depending only on $\theta$ \cite{Heilbronn}. Later, this proposition was improved to $\theta < \frac{4}{7}$ by Zaharescu \cite{Zah}.

Expanding beyond rational approximations with square or unrestricted denominators, if $Q$ is a polynomial with integer coefficients, positive for natural numbers, with $Q(0) = 0$, then the set of values obtained by applying $Q$ to natural numbers, denoted as $Q(\N)$, forms an $\mathcal{H}$-set \cite{Baker}. To add to this growing collection of literature, if $b \in \{2,3,\ldots\}$, we prove that the set %%  $\mathfrak{D}_{b}$,
\begin{align*}
\mathfrak{D}_{b} := \left\{\sum_{d=0}^{\infty}a_{d}b^{d}: a_{d}\in \left\{0,1\right\}, \ 0<\sum_{d=0}^{\infty}a_{d}<\infty\right\},
\end{align*}
 consisting of positive integers made of digits $0$ and $1$ in base $b$ representation, is a $\mathcal{H}-$set.  

Through elementary means, see Section~\ref{sec:elem}, it is possible to prove that for any $\gamma \in \R$ we have
\begin{align}
\label{eq:ElemBound}
\min_{1 \leq n \leq N, \ n \in \mathfrak{D}_{b}} \|\gamma n\| \leq   \frac{A_{b}}{\log N},
\end{align}
for some effective real constant $A_{b}$, depending only on $b$. Hence, we use~\eqref{eq:ElemBound} as a benchmark for further improvements. We enhance the decay function with the help of exponential sums as follows.
\begin{thm}
\label{main1}
For any $\gamma \in \R$, we have
\begin{align*}
\min_{1 \leq n \leq N, \ n \in \mathfrak{D}_{b}} \|\gamma n\| \leq \frac{C_{b}}{(\log N)^2},
\end{align*}
for some effective real constant $C_{b}$, depending only on $b$.
\end{thm}
Reverting attention to the case of squares,  researchers have conjectured that \eqref{Squareref} holds for $\theta < 1$ \cite{Heilbronn, Zah}. Inspired by the conjecture, we will say that a set of positive integers $S$ is a \textit{strong-approximating} set if for all $\theta < 1$ and $\gamma \in \R$, there exists a constant $C_{\theta}>0$ so that
\begin{align*}
\min_{1 \leq n \leq N}\|\gamma s_{n}\| \leq C_{\theta}N^{-\theta},
\end{align*}
where $s_{1}<s_{2}<\ldots$ are all the elements of $S$.

We prove that unlike the set of natural numbers, or possibly positive squares, the set $\mathfrak{D}_{b}$ is not strong-approximating when $b \geq 3$.
\begin{thm}
\label{nHeilbronn}
Let $b_{1}<b_{2}<\ldots$ be all the elements of $\mathfrak{D}_{b}$. For any $N \in \N$, there exists $\gamma_{N} \in \R$, so that 
\begin{align*}
\min_{1 \leq n \leq N}\|\gamma_{N} b_{n}\| \geq \frac{1}{b^4}N^{- \frac{\mathrm{\log}_{2} b}{b-1}}.
\end{align*}
\end{thm}
\section{Definitions and Conventions}
We list the following:
\begin{itemize}
\item $\N$ denotes the set $\{1,2,\ldots\}$.
\item $\gamma$ is a real parameter.
\item  $\lceil \gamma \rceil$ denotes the smallest integer greater than or equal to $\gamma$.
\item $\lfloor \gamma \rfloor$ denotes the largest integer less than or equal to $\gamma$.
\item $\{\gamma\}$ denotes the fractional part of gamma, defined as $\{\gamma\} = \gamma - \lfloor \gamma \rfloor$.
\item For $x \in \R$ it is designated that $e(x) := e^{2\pi i x}$.
\item $\ll$ is the Vinogradov symbol. That is, for any quantities $U$ and $V$, we have the following equivalent definitions (where $U$ and $V$ are defined):
$$U\ll V~\Longleftrightarrow~U=O(V)~\Longleftrightarrow~|U|\le c |V|,$$
for some constant $c>0$, which throughout the paper is allowed to depend on the parameter $b \in \{2,3,\ldots\}$, but not on the parameters $N \in \N$ and $\gamma \in \R$.
\end{itemize}

\section{Some elementary observations}
\label{sec:elem}
Let $t(b,N) = \lfloor \log_{b}(N(b-1)+1) - 1\rfloor$ so that $t(b,N)$ is the largest integer $t$ satisfying $1 +b+\ldots+b^t \leq N$, and put
\begin{align*}
P_{N,b} = \left\{s: s=\sum_{j=0}^{d}b^{j}, d\in \{0,\ldots,t(b,N)\} \right\}.
\end{align*}
For each $\gamma \in \R$ either there exists   $u \in P_{N,b}$  with $\| u \gamma\| \leq 1/ (t(b,N)+1) $ and~\eqref{eq:ElemBound} follows, or by the pigeonhole principle, there exists $u,v \in P_{N,b}$ and $h \in \{0,\ldots,t\}$ with $v>u$ so that 
\begin{align*}
\frac{h}{t(b,N)+1}\leq \{u \gamma\} \leq \{v\gamma\} \leq \frac{h+1}{t(b,N)+1},
\end{align*}
and thus $\|(v-u)\gamma\|\leq 1/ (t(b,N)+1) $, where it should be noted that 
\begin{align*}
(v-u) \in \mathfrak{D}_{b} \cap [1,N],
\end{align*} 
and~\eqref{eq:ElemBound} follows again. 

Before employing exponential sums to strengthen~\eqref{eq:ElemBound}, it is prudent to inquire whether a comprehensive application of the pigeonhole principle and the method of differencing could yield a result of comparable or greater strength than Theorem~\ref{main1}. After creating sufficient technology, we will see that the answer is no.
\begin{definition}
For a set $A,S \subseteq \mathbb{N}$, we define

\begin{align*}
D^{+}(A) &:= \{y-x: y,x \in A, y>x\}\\
M_{1}^{+}(S) &:= \max\{\#J: J \subseteq \mathbb{Z}, D^{+}(J)\subseteq S\}\\
M_{2}^{+}(S) &:= \max\{\#J: J \subseteq S, D^{+}(J)\subseteq S\}.
\end{align*}
\end{definition}
One observes that by a similar argument used to show \eqref{eq:ElemBound} (using the pigeonhole principle), we can show that for any non-empty set $S \subseteq \mathbb{N}$ and $\gamma \in \R$, one has
\begin{align*}
\min_{n \in S} \|\gamma n\| \leq \frac{1}{\max\{M_{1}^{+}(S)-1, M_{2}^{+}(S)+1\}}.
\end{align*}
By the trivial inequality $M_{2}^{+}(S) \leq M_{1}^{+}(S)$, if we manage to see that $M_{1}^{+}(\mathfrak{D}_{b} \cap \{1,\ldots,N\}) \ll \log N$ whenever $b \geq 3$ we could convince ourselves that such elementary methods are not enough to strengthen \eqref{eq:ElemBound} sufficiently.
\begin{lem}
For $b \geq 3$ we have
\begin{align*}
M_{1}^{+}(\mathfrak{D}_{b} \cap [1,N]) \ll \log N.
\end{align*}
\end{lem}
\begin{proof}
Suppose that $A(b,N) \subseteq \Z$ satisfies  $\#A(b,N) = M_{1}^{+}(\mathfrak{D}_{b} \cap [1,N])$ and $D^{+}(A(b,N)) \subseteq \mathfrak{D}_{b} \cap [1,N]$. We may list the elements of $A(b,N)$ in an increasing order as $v_{1}<\ldots<v_{z}$ and put
\begin{align*}
P = ((v_{j+1}-v_{j}))_{j=1}^{z-1},
\end{align*}
a sequence of $z-1$ differences.

We know that each element $v_{j+1}-v_{j}$ of $P$ is if the form 
\begin{align*}
\sum_{d=0}^{\lfloor\log_{b} N\rfloor} \xi_{d,j} b^{d},
\end{align*}
where $\xi_{d,j} \in \{0,1\}$ and $\xi_{d,j} \neq 0$ for some $d \in \{0,\ldots,\lfloor\log_{b} N\rfloor\}$.

We claim that for $1\leq j_{1} < j_{2} \leq z-1$ not both $\xi_{d,j_{1}}$ and $\xi_{d,j_{2}}$ are $1$; indeed for the sake of contradiction suppose that there exists the smallest value $a \in \{1,\ldots,z-2\}$, for which there exists $h_{a} \in \N$ with $1 \leq h_{a} \leq z-1-a$, $Y(a,h_{a}) \in \{0,\ldots,\lfloor\log_{b}(N)\rfloor\}$ and 
\begin{align*}
\xi_{Y(a,h_{a}),h_{a}} = \xi_{Y(a,h_{a}),h_{a}+a}=1.
\end{align*}
Then
\begin{align*}
v_{h_{a}+a}-v_{h_{a}}= \sum_{t=0}^{a-1}v_{h_{a}+t+1}-v_{h_{a}+t} = \sum_{d=0}^{\lfloor \log_{b}(N)\rfloor} \rho_{d}b^{d},
\end{align*}
where $\rho_{d} \in \{0,1,2\}$ and $\rho_{Y(a,h_{a})} = 2$, implying $v_{h_{a}+a}-v_{h_{a}} \not \in \mathfrak{D}_{b}$, which is impossible.

For each $j \in \{1,\ldots z-1\}$ put
\begin{align*}
K_{j} := \{d: \xi_{d,j} = 1\},
\end{align*}
so that $\#K_{j} \geq 1$ ,
\begin{align*}
\bigcup_{j=1}^{z-1}K_{j}\subseteq \{0,\ldots,\lfloor\log_{b} N\rfloor\}
\end{align*}
and $K_{j_{1}} \cap K_{j_{2}} = \emptyset$ if $j_{1} \neq j_{2}$. These constraints imply that 
\begin{align*}
z-1 \leq \lfloor\log_{b} N\rfloor+1,
\end{align*}
and therefore
\begin{align*}
M_{1}^{+}(\mathfrak{D}_{b}\cap[1,N]) \leq \lfloor\log_{b} N\rfloor+2.
\end{align*}
\par \vspace{-\baselineskip} \qedhere
\end{proof}
In our path towards the proof of Theorem \ref{main1}, we employ the set $\mathfrak{D}_{b}^{*}$ defined as
\begin{align*}
\mathfrak{D}_{b}^{*} :=  \mathfrak{D}_{b} \cup \left\{b^{d}-b^c: d,c \in \N \cup \{0\}, \ d > c\right\},
\end{align*}
and obtain the following result.
\begin{lem}
\label{main0}
For all $\gamma \in \R$, we have
\begin{align*}
\min_{1 \leq n \leq N, n \in \mathfrak{D}_{b}^{*}} \|\gamma n\| \ll \frac{1}{(\log N)^2},
\end{align*}
where the implied constant is effective and dependent only on $b$.
\end{lem}
Lemma \ref{main0} directly leads us to Theorem \ref{main1} due to the following assertion:
\begin{lem}
\label{Db*}
If for all $\gamma \in \R$ there is some positive function $f : \N \rightarrow \R$ so that 
\begin{align*}
\min_{1 \leq n \leq N, \ n \in \mathfrak{D}_{b}^{*}} \|\gamma n\| \leq f(N),
\end{align*}
then we have
\begin{align*}
\min_{1 \leq n \leq N, \ n \in \mathfrak{D}_{b}} \|\gamma n\| \leq (b-1)f(N).
\end{align*}
\end{lem}
\begin{proof}
Let $\gamma^{*} = \frac{\gamma}{b-1}$. Given the assumption 
\[
\min_{1 \leq n \leq N, \ n \in \mathfrak{D}_{b}^{*}}\|\gamma^{*}n\| \leq f(N),
\]
we have two cases to consider:
\begin{itemize}
\item If there exists $x \in \mathfrak{D}_{b}$ such that $1 \leq x \leq N$ and $\|x \gamma^{*}\| \leq f(N)$, then
\[
\|x \gamma\| \leq (b-1)f(N).
\]
\item If there exist positive integers $d$ and $c$ with $d > c$, and $1 \leq b^{d}-b^{c} \leq N$, such that $\|(b^{d}-b^{c}) \gamma^{*}\| \leq f(N)$, then $\|\frac{b^{d}-b^{c}}{b-1}\gamma \|\leq f(N)$, where $\frac{b^{d}-b^{c}}{b-1} \in \mathfrak{D}_{b}$ and $1 \leq \frac{b^{d}-b^{c}}{b-1} \leq N$, because $1 \leq b^{d}-b^{c} \leq N$ and $b^{d}-b^{c}$ is a multiple of $b-1$.
\end{itemize}
\par \vspace{-\baselineskip} \qedhere
\end{proof}
\section{Exponential sums with digit restrictions}
This section establishes upper bounds for exponential sums involving the parameter $\gamma \in \mathbb{R}$. Our focus is on cases where $\gamma$ exhibits unfavourable behaviour—defined as situations where, for a given $m \in \N$, $\|\gamma x \|$ exceeds $\frac{1}{2b^m}$ for all $x$ in the set $\mathfrak{D}_{b}^{*} \cap [1, N]$. If $\gamma$ is unfavourable for large values of $N$ (with hidden $m$), our analysis reveals improvements over the trivial upper bound for these sums. In the next section, such bounds on the exponential sums will be used alongside with the Erdős-Turán inequality to prove that $\gamma$ is ``not unfavourable'' when $N \geq J_{b}b^{m/2}$, where $J_{b}$ is an effective real constant depending only on $b$.

For $r \in \N\cup \{0\},$ define
\begin{equation*}
\mathfrak{D}_{b}(r) := \left\{\sum_{k=0}^{r}\xi_{k}b^{k}: \xi_{k} \in \{0,1\}\right\},
\end{equation*}
so that $\mathfrak{D}_{b}(r)$ has $2^{r+1}$ elements. Also, let
\begin{equation*}
\mathfrak{D}_{b}^{*}(r) := \mathfrak{D}_{b}(r)\cup \{b^{d}-b^{c}: c,d \in \left\{0,\ldots,r\}, d>c\right\}.
\end{equation*}
\begin{lem}
\label{Prodbound}
For $k \in \Z$, $r \in \N \cup \{0\}$ and $\gamma \in \R$ we have the bound
\begin{equation*}
\left|\sum_{j \in \mathfrak{D}_{b}(r)}e(kj\gamma)\right| \leq 2^{r+1} \prod_{d=0}^{r}(1 - \pi \|k b^d \gamma\|^2).
\end{equation*}
\end{lem}
\begin{proof}
We can write
\begin{align*}
\left|\sum_{j \in \mathfrak{D}_{b}(r)}e(kj\gamma)\right| &= \left|\prod_{d=0}^{r}(1 + e(kb^d \gamma))\right| \\
&= \left| \prod_{d=0}^{r}(e(kb^d \gamma/2)+e(-kb^d \gamma/2))\right|\\
&= \left| \prod_{d=0}^{r}2\cos(k b^d \gamma)\right|\\
&= 2^{r+1}\left| \prod_{d=0}^{r}\cos(k b^d \gamma)\right|.
\end{align*}
By Lemma \ref{Cosbound}, the quantity above has the upper bound of
\begin{equation*}
2^{r+1} \prod_{d=0}^{r}(1 - \pi \|k b^d \gamma\|^2).
\end{equation*}
\par \vspace{-\baselineskip} \qedhere
\end{proof}
\begin{lem}
\label{Addcomb}
Given $t$ integers $a_{1},\ldots,a_{t}$ distinct modulo $k$ for some $k \geq 1$, if $t \geq 3\sqrt{k}$ there exists some $h \geq 1$ so that there exists $h$ distinct elements from $\{a_{1},\ldots,a_{t}\}$ that sum to a number that is a multiple of $k$.
\end{lem}
\begin{proof}
We can apply \cite[Corollary 3.2.1]{Olsen} to the additive group $\Z/ k\Z$.
\end{proof}
\begin{lem}
\label{Addcomb2}
If for a fixed $\beta > 0$, $r \in \mathbb{N}\cup \{0\}$ and $\gamma \in \R$, one has $\|\gamma x\| > \beta$ for all $x \in \mathfrak{D}_{b}^{*}(r)$ with $x \neq 0$, then for each $k \in \N$ there are  $g < 3\sqrt{k}$ distinct values $d_{1}<\ldots<d_{g} \in \{0,\ldots,r\},$ for which $\|k b^{d_{j}}\gamma\| \leq \beta$.
\end{lem}
\begin{proof}
The assertion is trivially true for $k = 1$; assume $k \geq 2$. For the integers $d_{1},\ldots,d_{g}$ we can extract the integers $s_{1},\ldots,s_{g}$ so that $\|kb^{d_{j}}\gamma\| = |kb^{d_{j}}\gamma -s_{j}|$. Observe that if for $1 \leq u < v \leq g$ and $s_{u} \equiv s_{v} \mod k$ we will have that 
\begin{align*}
&|(k (b^{d_{v}}-b^{d_{u}})\gamma - s_{v}+s_{u}| \leq |kb^{d_{v}}\gamma -s_{v}|+|-(kb^{d_{u}}\gamma - s_{u})| \leq 2\beta.
\end{align*}
Hence,
\begin{align*}
|(b^{d_{v}}-b^{d_{u}})\gamma + \frac{-s_{v}+s_{u}}{k}| \leq \frac{2 \beta}{k} \leq \beta.
\end{align*}
Thus,
\begin{align*}
\|(b^{d_{v}}-b^{d_{u}})\gamma\| \leq \beta,
\end{align*}
which contradicts our assumption. Therefore the set $\{s_{1},\ldots,s_{g}\}$ holds $t$ distinct congruences modulo $k$.

Now observe that if $g > \frac{k}{2}$ there exists $s_{u_{1}},s_{u_{2}} \in \{s_{1},\ldots,s_{g}\}$ satisfying $s_{u_{1}}+s_{u_{2}} \equiv 0 \mod k$ and $u_{1}\neq u_{2}$. With these constraints, one would obtain that
\begin{align*}
|k(b^{d_{u_{1}}}+b^{d_{u_{2}}})\gamma-s_{u_{1}}-s_{u_{2}}| \leq |k b^{d_{u_{1}}}\gamma-s_{u_{1}}|+|k b^{d_{u_{2}}}\gamma-s_{u_{2}}| \leq 2\beta.
\end{align*}
Hence,
\begin{align*}
|(b^{d_{u_{1}}}+b^{d_{u_{2}}})\gamma - \frac{s_{u_{1}}+s_{u_{2}}}{k}| \leq \frac{2\beta}{k} \leq \beta.
\end{align*}
Thus,
\begin{align*}
\|((b^{d_{u_{1}}}+b^{d_{u_{2}}})\gamma)\| \leq \beta,
\end{align*}
which contradicts our assumption. Therefore the inequality $g \leq \frac{k}{2}$ holds, implying $g < 3\sqrt{k}$ for $1 \leq k \leq 35$. For the remainder of the argument, assume $k \geq 36$.

Note that if $g \geq 3\sqrt{k}$, by Lemma \ref{Addcomb} there exists $h \in \N$ with
\begin{align*}
1 \leq h \leq \lceil 3 \sqrt{k}\rceil,
\end{align*}
and an increasing sequence of integers satisfying 
\begin{align*}
1 \leq u_{1} < \ldots < u_{h} \leq \lceil 3 \sqrt{k}\rceil,
\end{align*}
so that $s_{u_{1}}+\ldots+s_{u_{h}} \equiv 0 \mod k$. Hence, we can write
\begin{align*}
|k(b^{d_{u_{1}}}+\ldots+b^{d_{u_{h}}})\gamma - s_{u_{1}}-\ldots-s_{u_{h}}| \leq \sum_{j=1}^{h}|k(b^{d_{u_{j}}})\gamma-s_{u_{j}}| \leq h \beta.
\end{align*}
Thus (with $k \geq 36$),
\begin{align*}
|(b^{d_{u_{1}}}+\ldots+b^{d_{u_{h}}})\gamma- \frac{s_{u_{1}}+\ldots+s_{u_{h}}}{k}| \leq \frac{h \beta}{k} \leq \frac{\lceil 3 \sqrt{k}\rceil \beta}{k} \leq \beta.
\end{align*}
Therefore
\begin{align*}
\|(b^{d_{u_{1}}}+\ldots+b^{d_{u_{h}}})\gamma\| \leq \beta.
\end{align*}
Thus, the inequality $g < 3\sqrt{k}$ must hold for $k \geq 36$ and $1 \leq k \leq 35$.
\end{proof}
For $t \geq 1$, we define the sets
\begin{align*}
G_{b}(t) := \left\{y: y \in \R, \frac{1}{2b^t} < \|y\| \leq \frac{1}{2b^{t-1}}\right\}.
\end{align*}
\begin{lem}
\label{IntegerDistanceLemma}
If $y \in G_{b}(t)$ for some $t \geq 2$, then $by \in G_{b}(t-1)$.
\end{lem}
\begin{proof}
Assume $y \in G_{b}(t)$ for some $t \geq 2$. This assumption implies the existence of an integer $c$ such that
\begin{align*}
\frac{1}{2b^{t}} < |y-c| \leq \frac{1}{2b^{t-1}}.
\end{align*}
Therefore, we can deduce
\begin{align*}
\frac{1}{2b^{t-1}} < |by-bc| \leq \frac{1}{2b^{t-2}}.
\end{align*}
Since $b \geq 2$ and $t \geq 2$, it follows that $\frac{1}{2b^{t-2}} \leq \frac{1}{2}$, which implies $\|by\| = |by-bc|$. Consequently, we obtain the inequality
\begin{align*}
\frac{1}{2b^{t-1}} < \|by\| \leq \frac{1}{2b^{t-2}}.
\end{align*}
\par \vspace{-\baselineskip} \qedhere
\end{proof}
\begin{lem}
\label{main-1}
Let $m$ be a natural number, $r \in \N \cup \{0\}$, $k \in \N$, and $\gamma \in \R$ such that $\|\gamma x\| > \frac{1}{2b^m}$ for all $x \in \mathfrak{D}_{b}^{*}(r)$, $x \neq 0$. Then we have the following bound:
\begin{align*}
 \left|\sum_{j \in \mathfrak{D}_{b}(r), \ j \neq 0} e(kj\gamma)\right| \leq 2^{r+3}\left(1 - \frac{\pi}{4b^2}\right)^{\frac{r-3\sqrt{k}+1}{m}}.
\end{align*}
\end{lem}
\begin{proof}
First, using Lemma \ref{Prodbound}, we obtain the bound:
\begin{align*}
\left|\sum_{j \in \mathfrak{D}_{b}(r)} e(kj\gamma)\right| \leq 2^{r+1} \prod_{d=0}^{r} \left(1 - \pi \|k b^d \gamma\|^2\right).
\end{align*}
Let
\begin{align*}
V_{\gamma, k, r}(m) = \left\{d : d \in \{0,\ldots,r\}, \|k b^d \gamma\| > \frac{1}{2b^m}\right\}.
\end{align*}
Using the product inequality above, we have the estimate:
\begin{align*}
\left|\sum_{j \in \mathfrak{D}_{b}(r)} e(kj\gamma)\right| \leq 2^{r+1} \prod_{d \in V_{\gamma, k, r}(m)} \left(1 - \pi \|k b^d \gamma\|^2\right).
\end{align*}
For each element $v$ in $V_{\gamma,k,r}(m)$, we have $k b^{v}\gamma \in G_{b}(h)$ for some $h \in \{1,\ldots,m\}$. Using Lemma \ref{IntegerDistanceLemma} and induction, we can show that $kb^{v+h-1}\gamma \in G_{b}(1)$. Let $0 \leq a_{1} < \ldots < a_{\#V_{\gamma,k,r}(m)} \leq r$ be all the elements of $V_{\gamma,k,r}(m)$ arranged in increasing order.

If $\#V_{\gamma,k,r}(m) \geq m$, then we can consider inequalities
\begin{align*}
a_{mt} < \ldots < a_{mt+m-1},
\end{align*}
for $t = 0,\ldots,\lfloor\frac{\# V_{\gamma,k,r}(m)}{m}\rfloor-1$. For each such $t$, there is an $h_{t} \in \{1,\ldots,m\}$ so that $kb^{a_{mt}+h_{t}-1}\gamma \in G_{b}(1)$, and thus $a_{mt}+h_{t}-1 \in V_{\gamma,k,r}(m)$. We have 
\begin{align*}
a_{mt} \leq a_{mt}+h_{t}-1 \leq a_{mt+m-1},
\end{align*}
and hence $a_{mt}+h_{t}-1 \in \{a_{mt},\ldots,a_{mt+m-1}\}$ as $a_{mt}+h_{t}-1 \in V_{\gamma,k,r}(m)$. Therefore, for $t = 0,\ldots,\lfloor\frac{\# V_{\gamma,k,r}(m)}{m}\rfloor-1$, the sets $\{a_{mt},\ldots,a_{mt+m-1}\}$ have an element $d_{t}$ for which $kb^{d_{t}} \gamma \in G_{b}(1)$. 

Consequently, there are at least $\lfloor\frac{\# V_{\gamma,k,r}(m)}{m}\rfloor$ elements $d$ in $V_{\gamma,k,r}(m)$ that satisfy $kb^d \gamma \in G_{b}(1)$. We can now establish an upper bound on:
\begin{align*}
\left|\sum_{j \in \mathfrak{D}_{b}(r)} e(kj\gamma)\right| &\leq 2^{r+1} \prod_{d \in V_{\gamma, k, r}(m)} \left(1 - \pi \|k b^d \gamma\|^2\right)\\
&\leq 2^{r+1} \prod_{d \in V_{\gamma, k, r}(m), \ kb^{d}\gamma \, \in \, G_{b}(1)} \left(1 - \pi \|k b^d \gamma\|^2\right)\\
&\leq 2^{r+1} \prod_{d \in V_{\gamma, k, r}(m), \ kb^{d}\gamma \, \in \, G_{b}(1)} \left(1 - \pi \left(\frac{1}{2b}\right)^2\right)\\
&\leq 2^{r+1} \left(1 - \frac{\pi}{4b^2}\right)^{\lfloor\frac{\# V_{\gamma,k,r}(m)}{m}\rfloor}.
\end{align*}
By inspection, we can verify that the above inequality holds even if $\# V_{\gamma,k,r}(m) < m$. Using Lemma \ref{Addcomb2}, we observe that 
\begin{align*}
\# V_{\gamma,k,r}(m) \geq r+1-(3\sqrt{k}) = r-3\sqrt{k}+1.
\end{align*}
Therefore, it follows that
\begin{align*}
&\left|\sum_{j \in \mathfrak{D}_{b}(r)} e(kj\gamma)\right| \leq 2^{r+1} \left(1 - \frac{\pi}{4b^2}\right)^{\lfloor\frac{r-3\sqrt{k}+1}{m}\rfloor}.
\end{align*}
In particular, we have:
\begin{align*}
&\left|\sum_{j \in \mathfrak{D}_{b}(r), \ j \neq 0} e(kj\gamma)\right| \leq  2^{r+3} \left(1 - \frac{\pi}{4b^2}\right)^{\frac{r-3\sqrt{k}+1}{m}}.
\end{align*}
\end{proof}
\section{Proof of Theorem \ref{main1}}

To make use of our bounds on exponential sums, we employ the Erd{\"o}s-Tur{\'a}n inequality, but before that, we will define the notion of discrepancy modified from \cite[p. 120]{Harman}.
\begin{definition}
Let $(x_{n})_{n=1}^{T}$ be a sequence of real numbers, we define
\begin{align*}
\mathcal{L}((x_{n})_{n=1}^{T}) = \sup_{\mathcal{I} \subset [0,1)}\left|\left(\sum_{n=1, \ \{x_{n}\}\in \mathcal{I}}^{T} 1\right) - T\lambda(\mathcal{I})\right|,
\end{align*}
where $\mathcal{I}$ denotes an interval (open, closed, or half-open) and $\lambda$ denotes the Lebesgue measure.
\end{definition}
With the above definition, we can state the Erd{\"o}s-Tur{\'a}n Theorem, here we modify \cite[Theorem 5.5]{Harman} (see also \cite{Erdos} for its historical origins).
\begin{lem}
\label{Discrepancy}
Let $(x_{n})_{n=1}^{T}$ be a sequence of reals, and let $G$ be a positive integer. With $C = (2 +\frac{2}{\pi})$ we have the inequality
\begin{align*}
\mathcal{L}((x_{n})_{n=1}^{T}) \leq \frac{T}{G+1}+C\sum_{k=1}^{G}\frac{1}{k}\left|\sum_{n=1}^{T}e(kx_{n})\right|.
\end{align*}
\end{lem}

\begin{lem}
\label{Necessary1.1}
For all $\gamma \in \R$ and $m \in \N \cup \{0\}$, there exists an effective constant $J_{b} > 0$ (depending only on $b$), such that if $w$ is an integer with $w \geq J_{b}b^{m/2}$ then $\|\gamma y\| \leq \frac{1}{2b^{m}}$ for some $y\in \mathfrak{D}_{b}^{*}(w)$, $y \neq 0$.
\end{lem}
\begin{proof}
Let all the elements of $\mathfrak{D}_{b}$ be arranged as $1=b_{1}<b_{2}<\ldots$. We can observe that when $r \in \N \cup \{0\}$ we have
\begin{align*}
\mathfrak{D}_{b}(r) = \{0\}\cup\{b_{1},\ldots,b_{2^{r+1}-1}\}.
\end{align*}
Let $\mathcal{I}_{m} = [0, \frac{1}{2b^{m}}]$, so that $\lambda(\mathcal{I}_{m}) = \frac{1}{2b^m}$, and assume that $r$ is a non-negative integer so that $\|\gamma x\| > \frac{1}{2b^m}$ for all $x \in \mathfrak{D}_{b}^{*}(r)$, $x \neq 0$.  With Lemma \ref{Discrepancy} (using $G = 8b^m$)  we have
\begin{align*}
\left|\left(\sum_{n=1, \ \{b_{n}\gamma\}\in \mathcal{I}_{m}}^{2^{r+1}-1} 1\right) - (2^{r+1}-1)\frac{1}{2b^m}\right| \leq \frac{2^{r+1}-1}{8b^m + 1}+C\sum_{k=1}^{8b^m}\frac{1}{k}\left|\sum_{n=1}^{2^{r+1}-1}e(kb_{n}\gamma)\right|.
\end{align*}
With the above inequality (and the assumption on $r$), it is deduced that

\begin{align*}
0=\sum_{n=1, \{\gamma b_{n}\} \in \mathcal{I}_{m}}^{2^{r+1}-1}1 &\geq \frac{2^{r+1}-1}{2b^m}-\left( \frac{2^{r+1}}{8b^m} + C \sum_{k=1}^{8b^m}\frac{1}{k}\left|\sum_{n=1}^{2^{r+1}-1}e(k\gamma b_{n})\right|\right)\\
&\geq \frac{2^{r-1}}{b^m}-\frac{2^{r-2}}{b^m}- C\sum_{k=1}^{8b^m}\left|\sum_{n=1}^{2^{r+1}-1}e(k\gamma b_{n})\right|\\
&= \frac{2^{r-2}}{b^m}-C\sum_{k=1}^{8b^m}\left|\sum_{n=1}^{2^{r+1}-1}e(k\gamma b_{n})\right|.
\end{align*}
Using Lemma \ref{main-1} it is observed that 
\begin{align}
\label{NecEqn}
\frac{2^{r-2}}{b^m} \leq C \sum_{k=1}^{8b^m}2^{r+3}\left(1 - \frac{\pi}{4b^2}\right)^{\frac{r-3\sqrt{k}+1}{m}}.
\end{align}
Inequality \eqref{NecEqn} implies that
\begin{align}
\label{NecEqn1}
\frac{1}{b^m}\leq C(256b^m)\left(1 - \frac{\pi}{4b^2}\right)^{\frac{r-3\sqrt{8b^m}+1}{m}}.
\end{align}
Put $U = \frac{4b^2}{4b^2-\pi}$, so that with \eqref{NecEqn1} we have the inequality
\begin{align*}
r \leq 3\sqrt{8b^m}+2m^2\log_{U}(256Cb)-1.
\end{align*}
Hence, $r \leq H_{b}b^{m/2}$ for an effective positive constant $H_{b}$ depending only on $b$.

To rephrase our above steps, we have shown that if $r$ is a non-negative integer satisfying $\|\gamma x\| > \frac{1}{2b^m}$ for all positive $x$ with $x \in \mathfrak{D}_{b}^{*}(r)$, then $r \leq H_{b}b^{m/2}$. Taking $J_{b} = 2H_{b}$ proves the Lemma.
\end{proof}
For a given natural $N \in \N$ we recall that 
\begin{align*}
t(b,N) = \lfloor \log_{b}(N(b-1)+1) - 1\rfloor,
\end{align*}
so that $t(b,N)$ is the largest integer $t$ satisfying $1 +b+\ldots+b^t \leq N$. We observe that $\mathfrak{D}_{b}^{*}(t(b,N)) \subseteq \{0,\ldots,N\}$. Put $m_{N} = ~\lfloor 2\log_{b}(\frac{t(b,N)}{ J_{b}})\rfloor$, so that 
\begin{align*}
t(b,N) = J_{b}b^{\left(2\log_{b}(\frac{t(b,N)}{ J_{b}})\right)/2} \geq J_{b}b^{m_{N}/2}.
\end{align*}
By Lemma \ref{Necessary1.1} there exists $y \in \mathfrak{D}_{b}^{*}(t(b,N))$ with $y \neq 0$ so that 
\begin{align*}
\|\gamma y\| &\leq \frac{1}{2b^{m_{N}}}\\
&\leq \frac{b}{2b^{2\log_{b}(\frac{t(b,N)}{ J_{b}})}}\\
&= \frac{bJ_{b}^2}{2 t(b,N)^2}\\
&=\frac{bJ_{b}^2}{2  \lfloor \log_{b}(N(b-1)+1) - 1\rfloor^2}.
\end{align*}
To put it differently, we can write
\begin{align*}
\min_{1 \leq n \leq N, \ n \in \mathfrak{D}_{b}^{*}} \leq \frac{bJ_{b}^2}{2  \lfloor \log_{b}(N(b-1)+1) - 1\rfloor^2},
\end{align*}
implying Lemma \ref{main0}. With Lemmas \ref{main0} and \ref{Db*} we have Theorem \ref{main1}.

\section{Proof of Theorem \ref{nHeilbronn}}
For $t \in \N$, let 
\begin{align*}
\mathcal{Z}_{b,t} := \left\{\sum_{d=1}^{t} b^{u_{d}} : u_{d} \in \N\cup\{0\}\right\},
\end{align*}
so that $\mathcal{Z}_{b,t}$ contains sums of $t$ integer powers of $b$. Let $1 = b_{1}<b_{2}<\ldots$ be all the elements of $\mathfrak{D}_{b}$ arranged in an increasing order. It is inspected that with $T \in \N$,
\begin{align}
\label{RegSubset}
\{b_{1},\ldots,b_{2^{T}-1} \}\subseteq \bigcup_{t=1}^{T}\mathcal{Z}_{b,t}.
\end{align}
We now construct preparatory tools to understand the behaviour of $n \mod b^{k}-1$ as $n$ ranges through $\mathcal{Z}_{b,t}$ when $k$ is a positive integer.
\begin{lem}
\label{PrepnHeil}
Let $k \in \N$, and suppose that $u_{0},\ldots u_{k-1}$ are non-negative integers satisfying $u_{0}+\ldots+u_{k-1} > 0$. There exists integers $v_{0},\ldots,v_{k-1}$ with $v_{j} \in \{0,\ldots,b-1\}$ so that
\begin{align*}
\sum_{j=0}^{k-1}u_{j}b^{j} \equiv \sum_{j=0}^{k-1}v_{j}b^{j} \bmod (b^{k}-1)
\end{align*}
and
\begin{align*}
 0<\sum_{j=0}^{k-1}v_{j}\leq \sum_{j=0}^{k-1}u_{j}.
\end{align*}
\end{lem}
\begin{proof}
Suppose that $n$ is the smallest natural number such that there exist $u_{0},\ldots u_{k-1}$ with $\sum_{j=0}^{k-1}u_{j} = n$, for which the proposition is false. With this choice of $(u_{j})_{j=0}^{k-1}$ there exists some $d \in \{0,\ldots,k-1\}$ for which $u_{d} \geq b$ (and thus $n \geq b$). Let $r_{0},\ldots,r_{k-1}$ be non-negative integers so that $r_{d} = u_{d}-b$,
\begin{align*}
r_{d+1 - k \lfloor\frac{d+1}{k}\rfloor} = \left(u_{d+1 - k \lfloor\frac{d+1}{k}\rfloor}\right)+1
\end{align*}
and $r_{t} = u_{t}$ when $t \neq d$ and $ t \neq d+1 - k \lfloor\frac{d+1}{k}\rfloor$.

We observe that $(r_{j})_{j=0}^{k-1}$ is a non-negative integer sequence,
\begin{align*}
\sum_{j=0}^{k-1}u_{j}b^{j} \equiv \sum_{j=0}^{k-1}r_{j}b^{j} \bmod (b^{k}-1),
\end{align*}
and
\begin{align*}
 \sum_{j=0}^{k-1}r_{j} = \sum_{j=0}^{k-1}u_{j} -(b-1) = n-b+1.
\end{align*}
Thus, $n$ would not be the smallest number for which the proposition is false, yielding a contradiction.
\end{proof}
Let $k,t \in \N$, and put 
\begin{align*}
\mathfrak{A}(b,k,t) := \left\{\sum_{j=0}^{k-1}a_{j}b^{j}: a_{j}\in\{0,\ldots,b-1\}, 0<\sum_{j=0}^{k-1}a_{j} \leq t \right\}.
\end{align*}
\begin{lem}
\label{PrepnHeil2}
Let $t,k \in \N$, if $w \in \mathcal{Z}_{b,t}$ then there exists $c_{w} \in \mathfrak{A}(b,k,t)$ for which 
\begin{align*}
w \equiv c_{w} \bmod(b^{k}-1).
\end{align*}
\end{lem}
\begin{proof}
From the definition of the set $\mathcal{Z}_{b,t}$ there exist non-negative integers $r_{1},\ldots,r_{t}$ for which $w = \sum_{d=1}^{t}b^{r_{d}}$. By modular arithmetic, we gather that
\begin{align*}
w \equiv \sum_{j=0}^{k-1}u_{j}b^{j} \bmod(b^{k}-1),
\end{align*}
where $u_{0},\ldots,u_{k-1}$ are non-negative integers that satisfy $u_{0}+\ldots+u_{k-1} = t$. By Lemma \ref{PrepnHeil} there exist integers $v_{0},\ldots,v_{k-1}$ with $v_{j} \in \{0,\ldots,b-1\}$, $v_{0}+\ldots+v_{k-1} \leq t$, and
\begin{align*}
\sum_{j=0}^{k-1}u_{j}b^{j} \equiv \sum_{j=0}^{k-1}v_{j}b^{j} \bmod(b^{k}-1).
\end{align*}
Now observe that $\sum_{j=0}^{k-1}v_{j}b^{j} \in \mathfrak{A}(b,k,t)$.
\end{proof}
\begin{lem}
\label{PrepnHeil3}
Let $k,t \in \N$ with $k > \frac{t}{(b-1)}$, then $\mathcal{Z}_{b,t}$ does not contain multiples of $b^{k}-1$.
\end{lem}
\begin{proof}
It is recorded that 
\begin{align*}
\mathfrak{A}(b,k,t) \subseteq \{1,\ldots,h(b,k,t)\},
\end{align*}
where $h(b,k,t) \leq \sum_{j=0}^{k-1}(b-1)b^{j} = b^{k}-1$. We will now more carefully estimate $h(b,k,t)$. There are two cases to consider:
\begin{itemize}
\item If $t < b-1$, then $b \geq 3$ (as $t \in \N$), and $h(b,k,t) = tb^{k-1} < b^{k}-1$.
\item If $t \geq b-1$, then 
\begin{align*}
h(b,k,t) = \left(\sum_{d=0}^{\lfloor\frac{t}{(b-1)}\rfloor-1}(b-1)b^{k-1-d}\right) + \left(t-(b-1)\lfloor\frac{t}{b-1}\rfloor\right)b^{k-1-\lfloor\frac{t}{(b-1)}\rfloor}.
\end{align*}
Upon inspection, we note that
\begin{align*}
h(b,k,t) < \sum_{d=0}^{k-1}(b-1)b^{d} = b^{k}-1.
\end{align*}
\end{itemize}
Hence, we gather that (with the constraints on $t$ and $k$)
\begin{align*}
\mathfrak{A}(b,k,t) \subseteq \{1,\ldots,b^{k}-2\}.
\end{align*}
Since $\mathfrak{A}(b,k,t)$ does not contain multiples of $b^{k}-1$, by Lemma \ref{PrepnHeil2}, $\mathcal{Z}_{b,t}$ does not contain multiples of $b^{k}-1$. 
\end{proof}
Let $T = \lceil \log_{2}(N+1)\rceil$ so that
\begin{align*}
\{b_{1},\ldots,b_{N}\} \subseteq \{b_{1},\ldots,b_{2^{T}-1}\}
\end{align*}
By the subset relation \eqref{RegSubset} and Lemma \ref{PrepnHeil3}, $\{b_{1},\ldots,b_{2^{T}-1}\}$ does not contain multiples of $b^{k} - 1$ when $k = \lceil\frac{T}{b-1}\rceil + 1$. In particular it is verified that with $\gamma_{N} = \frac{1}{b^{k}-1}$,  $\|\gamma_{N} w\| \geq \frac{1}{b^{k}-1}$ for all $w \in \{b_{1},\ldots,b_{N}\}$. Observe that
\begin{align*}
\frac{1}{b^{k}-1} &\geq \frac{1}{b^{\frac{T}{b-1}+2}-1}\geq \frac{1}{b^{\frac{T}{b-1}+2}}\geq \frac{1}{b^{\frac{\log_{2}(N+1)}{b-1}+3}}\geq \frac{1}{b^{\frac{\log_{2}(N)}{b-1}+4}}= \frac{1}{b^4}N^{-\frac{\log_{2}b}{b-1}}.
\end{align*}
\section{Conjectures}
It is possible that the decay function in Theorem \ref{main1} is not optimal. Inspired by the form of the decay function presented in \eqref{Squareref}, employed to establish that the set of positive squares is a $\mathcal{H}-$set, we pose the following question.
\begin{conj}
There exists a constant $c > 0$ so that for all $\gamma \in \R$ we have
\begin{align*}
\min_{1 \leq n \leq N, \ n \in \mathfrak{D}_{b}} \|\gamma n\| \ll N^{-c}
\end{align*}
where the implied constant is independent of $\gamma$.
\end{conj}
With Theorem \ref{nHeilbronn} and a density argument we know that $c \leq \frac{1}{b-1}$. Next, once again turning our attention to the literature on approximating real numbers with rationals having square denominators, we present a further inquiry.
\begin{conj}
The set $\mathfrak{D}_{b,2}$, comprising the squares of elements from $\mathfrak{D}_{b}$, formally defined as
\begin{align*}
\mathfrak{D}_{b,2} := \{s^2 : s \in \mathfrak{D}_{b}\},
\end{align*}
is a $\mathcal{H}-$set.
\end{conj}
\section{Acknowledgements}
The author would like to thank Igor Shparlinski for suggesting the paper \cite{Olsen}, which eventually inspired the construction of $\mathfrak{D}_{b}^{*}$. During the preparation of this work, S.I. was supported by an Australian Government Research Training Program (RTP) Scholarship.
 
\newpage
\appendix
\section{An inequality with cosine}
\begin{lem}
\label{Cosbound}
If $x \in \R$, then
\begin{equation*}
|\cos(\pi x)| \leq 1- \pi \|x\|^2.
\end{equation*}
\end{lem}
\begin{proof}
For $x$ in suitable real intervals, put
\begin{align*}
f_{1}(x) = x- \tan(x)
\end{align*}
and
\begin{align*}
f_{2}(x) = \frac{\sin(x)}{x}.
\end{align*}
Note that $f_{1}(0) = 0$ and
\begin{align*}
f_{1}'(x) = 1 - \frac{1}{\cos^2(x)}
\end{align*}
 so that $f_{1}(x) \leq 0$ on the interval $[0,\frac{\pi}{2})$. We compute
\begin{align*}
f_{2}'(x) &= \frac{\cos(x)}{x^2}(x-\tan(x))\\
&= \frac{\cos(x)}{x^2}f_{1}(x),
\end{align*}
so that $f_{2}'(x) \leq 0$ on the interval $(0,\frac{\pi}{2})$. Hence on the interval $(0, \frac{\pi}{2})$ the inequality
\begin{align*}
f_{2}(x) = \frac{\sin(x)}{x} \geq \frac{\sin(\frac{\pi}{2})}{\frac{\pi}{2}} = \frac{2}{\pi}
\end{align*}
holds. By continuity, it must hold true that for all $y \in [0, \frac{\pi}{2}]$ we have
\begin{align*}
-\sin(y) \leq -\frac{2}{\pi}y.
\end{align*}
As $\frac{d}{dy}(\cos(y)) = - \sin(y)$, whenever $x \in [0,\frac{\pi}{2}]$ we must have
\begin{align*}
\cos(x) &= \cos(0) + \int_{0}^{x}-\sin(y) \ dy\\
&\leq 1 +\int_{0}^{x}-\frac{2}{\pi}y \ dy\\
&=1-\frac{x^2}{\pi}.
\end{align*}
Hence, whenever $x \in [0,\frac{1}{2}]$, we have the compound inequality
\begin{align*}
0\leq \cos(\pi x) \leq 1- \pi x^2.
\end{align*}
Now suppose that $x$ is some fixed real number (not necessarily in $[0,\frac{1}{2}]$), there would be an integer $n \in \Z$ so that $\|x\| = |x+n|$. Observe that
\begin{align*}
1 - \pi \|x\|^2 &\geq \cos(\pi \|x\|)\\
&=|\cos(\pi \|x\|)|\\
&=|\cos(\pi|x+n|)|\\
&=|\cos(\pi(x+n))|\\
&=|\cos(\pi x)\cos(\pi n)|\\
&=|\cos(\pi x)|.
\end{align*}
\end{proof}

\end{document}